\documentclass[10pt]{article}
\usepackage{amsfonts}
\usepackage{amsmath}
\usepackage{mathrsfs,url}
\usepackage{mathrsfs, amscd,amssymb,amsthm,amsmath,bm,graphicx,psfrag,subfigure}

\makeatletter

\renewcommand{\@seccntformat}[1]{{\csname the#1\endcsname}{\normalsize .}\hspace{.5em}}
\makeatother

\def \[{\begin{equation}}
\def \]{\end{equation}}
\newtheorem{thm}{Theorem}[section]

\newtheorem{defi}{Definition}
\newtheorem{claim}{Claim}
\newtheorem{fac}{Fact}
\newtheorem{lem}[thm]{Lemma}

\newenvironment{kst}
{\setlength{\leftmargini}{2\parindent}
 \begin{itemize}
 \setlength{\itemsep}{-1.1mm}}
{\end{itemize}}

\newenvironment{wst}
{\setlength{\leftmargini}{1.5\parindent}
 \begin{itemize}
 \setlength{\itemsep}{-1.1mm}}
{\end{itemize}}

\begin{document}
\setlength{\baselineskip}{16pt}
\begin{center}{\Large \bf Enumerating the total number of subtrees of trees\footnote{Financially supported by
the National Natural Science Foundation of China (Grant No. 11071096) and the Special Fund for Basic Scientific Research of Central Colleges (CCNU11A02015).}}

\vspace{4mm}

{\large Shuchao Li\footnote{E-mail: lscmath@mail.ccnu.edu.cn (S.C.
Li), wang06021@126.com (S.J. Wang)},\ Shujing Wang}\vspace{2mm}

{\small Faculty of Mathematics and Statistics,  Central China Normal
University, Wuhan 430079, P.R. China}\vspace{1mm}
\end{center}

\noindent {\bf Abstract}: Over some types of trees with a given number of vertices, which trees
minimize or maximize the total number of subtrees or leaf containing subtrees are studied. Here are some of the main results:\ (1)\, Sharp upper bound on the total number of subtrees (resp. leaf containing subtrees) among $n$-vertex trees with a given matching number is determined; as a consequence, the $n$-vertex tree with domination number $\gamma$ maximizing the total number of subtrees (resp. leaf containing subtrees) is characterized.  (2)\, Sharp lower bound on the total number of leaf containing subtrees among $n$-vertex trees with maximum degree at least $\Delta$ is determined; as a consequence the $n$-vertex tree with maximum degree at least $\Delta$ having a perfect matching minimizing the total number of subtrees (resp. leaf containing subtrees) is characterized. (3)\, Sharp upper  (resp. lower) bound on the total number of leaf containing subtrees among the set of all $n$-vertex trees with $k$
leaves (resp. the set of all $n$-vertex trees of diameter $d$) is determined.

\vspace{2mm} \noindent{\it Keywords}: Subtrees;
Leaves; Matching number; Domination number; Diameter;

\vspace{2mm}


\vspace{4mm}

\section{\normalsize Introduction}
We consider only simple connected graphs (i.e. finite, undirected graphs
without loops or multiple edges).  Let $G=(V_G, E_G)$ be a graph
with $u,v\in V_G$, $d_G(u)$ (or $d(u)$ for short) denotes the
degree of $u$; the \textit{distance} $d_G(u,v)$ is defined as the length of the shortest path between
$u$ and $v$ in $G$; $D_G(v)$ (or $D(v)$ for short) denotes the sum
of all distances from $v$. The \textit{eccentricity} $\varepsilon(v)$ of a
vertex $v$ is the maximum distance from $v$ to any other vertex. Vertices of minimum eccentricity form the \textit{center} (see \cite{15}). A tree $T$ has exactly one or two adjacent center vertices. In what follows, if a tree has a bicenter, then our considerations apply to any of its center vertices.

A subset $S$ of $V_G$ is called a \textit{dominating set} of $G$ if for every vertex $v \in V_G\setminus S$, there exists a vertex $u \in S$
such that $v$ is adjacent to $u$. A vertex in the dominating set is called a \textit{dominating vertex}. For a dominating
set $S$ of graph $G$ with $v \in S$ and $u \in V_G\setminus S$, if $vu \in E_G$, then $u$ is said to be dominated by $v$. The
\textit{domination number} of graph $G$, denoted by $\gamma(G)$, is defined as the minimum cardinality of dominating
sets of $G$. For a connected graph $G$ of order $n$, Ore \cite{9} obtained that $\gamma(G)\leqslant \frac{n}{2} $. And the equality case
was characterized independently in \cite{3,13}.
Given a graph $G$, the \textit{matching number} of $G$ is the cardinality
of one of its maximum matchings.

Throughout the text we denote by $P_n,\, K_{1,n-1}$ the path and
star on $n$ vertices, respectively. $G-v,\, G-uv$ denote
the graph obtained from $G$ by deleting vertex $v \in V_G$, or edge
$uv \in E_G$, respectively (this notation is naturally extended if
more than one vertex or edge is deleted). Similarly,
$G+uv$ is obtained from $G$ by adding edge $uv \not\in E_G$. For $v\in V_G,$ let
$N_G(v)$ (or $N(v)$ for short) denote the set of all the adjacent vertices of $v$ in $G.$
The \textit{diameter} diam$(G)$ of graph $G$ is the maximum
eccentricity of any vertex in $G$.
We refer to vertices of degree 1 of a tree $T$ as \textit{leaves}
(or \textit{pendants}), and the edges incident to leaves are called \textit{pendant edges}. The unique path connecting two vertices $v, u$ in $T$
will be denoted by $P_T(v, u)$.

Let
$$
W(T)=\frac{1}{2}\sum_{v\in V_T}D_T(v)
$$
denote the \textit{Wiener index} of $T,$ which is the sum of distances of all unordered pairs of
vertices. This topological index was introduced by Wiener \cite{14}, which has been one of the most widely used descriptors in quantitative structure-activity relationships. Since the majority of the chemical applications of the Wiener index deals with chemical compounds with acyclic molecular
graphs, the Wiener index of trees has been extensively studied over the past years; see \cite{15,3G,4,5,24} and the references there for
details.


Given a tree $T$, a \textit{subtree} of $T$ is just a connected induced subgraph of $T$. The number of subtrees as well the related subjects
has been studied. Let $T$ denote an $n$-vertex tree each of whose non-pendant vertices has degree at least three, Andrew and Wang \cite{17} showed that the average number of vertices in the subtrees of $T$ is at least $\frac{n}{2}$ and strictly less than $\frac{3n}{4}$. Sz\'ekely and Wang \cite{20} characterized the binary tree with $n$ leaves that has the greatest number of subtrees. Kirk and Wang \cite{19} identified the tree, given a size and maximum vertex degree, which
has the greatest number of subtrees.  Sz\'ekely and Wang \cite{21} gave a formula for the maximal number of subtrees a binary tree can possess over a given number of vertices.  They also showed that caterpillar trees (trees containing a path such that each vertex not belonging to the path is adjacent to a vertex on the path) have the smallest number of subtrees among binary trees. Yan and Ye \cite{25} characterized the tree with the diameter at least $d$, which has the maximum number of subtrees, and they characterized the tree with the maximum degree at least $\Delta$, which has
the minimum number of subtrees. Consider the collection of rooted labeled trees with $n$ vertices, Song \cite{S-C-W} derived a closed formula for the number of these trees in which the child of the root with the smallest label has a total of $p$ descendants. He also derived a recurrence relation for the number of these trees with the property that for each non-terminal vertex $v$, the child of $v$ with the smallest label has no descendants. The authors \cite{L-W-J} here determined the maximum (resp. minimum) value of the total number of subtrees of trees
among the set of all $n$-vertex trees with given number of leaves and characterize the extremal graphs.\ As well we determined
the maximum (resp. minimum) value of the total number of subtrees of trees with a given bipartition, the corresponding extremal graphs are characterized. For some related results on the enumeration of subtrees of trees, the reader is referred to Sz\'{e}kely and Wang \cite{02,03} and Wang \cite{12}.

It is well known that the Wiener index is maximized by the path and minimized by the star among general trees with
the same number of vertices. 
It is interesting that the Wiener index and the total number of subtrees of a tree share exactly the same extremal structure (i.e. the tree that maximizes/minimizes the corresponding index) among trees with a given number of vertices and maximum degree, although the values of the indices are in no general functional correspondence. On the other hand, an acyclic molecule can be expressed by a tree in
quantum chemistry (see \cite{3G}). Obviously, the number of subtrees of a tree can be regarded as a topological index.
Hence, Yan and Ye \cite{25} pointed out that to explore the role of the total number of subtrees in quantum chemistry is an interesting topic.  As a continuance of those works in \cite{19,L-W-J,20,21,17,25} which studied the correlations between the Wiener index and the number of subtrees of trees, in this paper we continue to characterize the extremal tree among some types of trees which minimizes or maximizes the total number of subtrees. Through a similar approach, we also identify the extremal trees that maximize (minimize) the number of leaf containing
subtrees.

\section{\normalsize Preliminaries}
Given a tree $T$ on $n$ vertices. Let $\mathscr{S}(T)$ denote the set of subtrees of $T$. For two fixed vertices $u, v$ in $V_T$, denote by $\mathscr{S}(T; u)$ (resp. $\mathscr{S}(T; u, v)$) the set of all subtrees of $T$, each of which contains $u$ (resp. $u$ and $v$). Let $\mathscr{S}^*(T)$ denote the set of all subtrees of $T$ each of which contains at least one leaf in $T$. Given a vertex $w$ in $V_T$, denote by $\mathscr{S}^*(T; w)$ the set of all subtrees of $T$ each of which contains $w$ and at least one leaf of $T$ different from $w$. For convenience, we call the subtree that contains at least one \textit{leaf of $T$ leaf containing subtree}. Set
$
F(T)=|\mathscr{S}(T)|, f_T(v)=|\mathscr{S}(T; v)|, f_T(v_i*v_j)=|\mathscr{S}(T; v_i, v_j)|, F^*(T)=|\mathscr{S}^*(T)|, f_T^*(v)=|\mathscr{S}^*(T; v)|.
$
Let $PV(T)$ be the set of leaves of $T$; it is routine to check the following fact.
\begin{fac}
Given a tree $T$, then $H(T):=T-PV(T)$ is a tree and $F^*(T)=F(T)-F(H)$.
\end{fac}
\begin{lem}[\cite{21}]\label{lem2.0}
Among trees on $n\geqslant 3$ vertices, the path $P_n$ minimizes $F^*$ with $F^*(P_n)= 2n-1$; while the star $K_{1,n-1}$ maximizes $F^*$ with $F^*(K_{1,n-1}) = 2^{n-1}+n-2$.
\end{lem}
\begin{figure}[h!]
\begin{center}
\psfrag{a}{$x_1$}\psfrag{A}{$X_1$}
\psfrag{b}{$x_n$}\psfrag{B}{$X_n$}
\psfrag{c}{$y_n$}\psfrag{C}{$Y_n$}
\psfrag{d}{$y_1$}\psfrag{D}{$Y_1$}
\psfrag{y}{$y$}\psfrag{Y}{$Y$}\psfrag{z}{$z$}\psfrag{Z}{$Z$}\psfrag{x}{$x$}\psfrag{X}{$X$}
  \includegraphics[width=110mm]{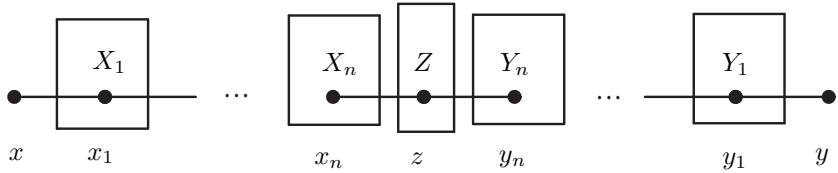}\\
  \caption{Path $P_W(x, y)$ connecting vertices $x$ and $y$. }
\end{center}
\end{figure}
Consider the tree $W$ in Fig. 1 with $x,\, y\in PV(W)$, and
$$
P_W(x, y) = xx_1 \ldots x_nzy_n \ldots y_1y(xx_1 \ldots x_ny_n \ldots y_1y)
$$
if $d_W(x, y)$ is even (odd) for any $n \geqslant 0$.
After the deletion of all the edges of $P_W(x, y)$ from $W$, some connected components
will remain. Let $X_i $ denote the component that contains $x_i $, let
$Y_i$ denote the component that contains $y_i$, for $i = 1, 2, \ldots , n$, and let $Z$
denote the component that contains $z$. (Note
that $z$ and $Z$ exist if and only if $d_W(x, y)$ is even.)
\begin{lem}[\cite{12}]\label{lem2.2}
In the above situation, if $f_{X_i} (x_i) \geqslant f_{Y_i} (y_i)$ and $f_{X_i}^* (x_i) \geqslant f_{Y_i}^* (y_i)$ for $i =1, \ldots , n$, then
\[\label{eq:2-1}
f_W(x) \geqslant f_W(y)
\]
and
\[\label{eq:2-2}
  f_W^*(x) \geqslant f_W^*(y).
\]
Furthermore, if a strict inequality $f_{X_i} (x_i) \geqslant f_{Y_i} (y_i)$ holds for any $i, i\in\{1,2\ldots,n\}$, then we have the strict inequalities in (\ref{eq:2-1}) and (\ref{eq:2-2}).
\end{lem}
\begin{lem}\label{lem2.1}
Let $T'$ be a graph obtained from a tree $T$ by deleting one leaf. Then $F(T') < F(T)$ and $F^*(T')<F^*(T).$ Furthermore, we have
$f_{T'}(v)<f_T(v)$ and $f_{T'}^*(v)\leqslant f_T^*(v)$ for any vertex $v$ in $V_{T'}$, with equality if and only if $T$ is a path and $v$ is the other leaf of $T$.
\end{lem}
\begin{proof}
It is straightforward to check that this result is true. We omit the procedure here.
\end{proof}

If we have a tree $T$ with $x$ and $y$ in $V_T$, and a rooted
tree $X$ that is not a single vertex, then we can build two new trees, first $T'$, by
identifying the root of $X, u$ with $x$,
second $T''$, by identifying the root of $X$, $u$ with $y$ (as depicted in Fig. 2).
\begin{figure}[h!]
\begin{center}
\psfrag{a}{$T'$}\psfrag{b}{$T''$}\psfrag{T}{$T$}
\psfrag{y}{$y$}\psfrag{Y}{$Y$}\psfrag{x}{$x$}\psfrag{X}{$X$}
  \includegraphics[width=90mm]{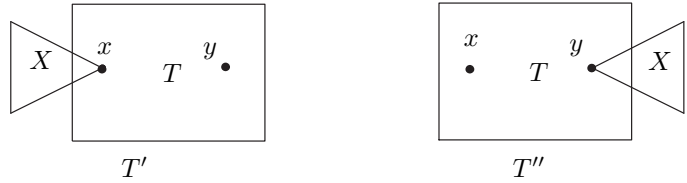}\\
  \caption{Trees $T'$ and $T''$.}
\end{center}
\end{figure}
\begin{lem}\label{lem2.3}
In the above situation, if $f_T(x)> f_T(y)$, then we have $F(T')> F(T'')$. Further more, if $x$ is not a leaf of $T$ and $f_T^*(x)\geqslant f_T^*(y)$, we have  $F^*(T')>F^*(T'')$. And if both $x$ and $y$ are leaves of $T$ and $f_T^*(x)\geqslant f_T^*(y)$, we also have  $F^*(T')\geqslant F^*(T'')$, with equality if and only if both $x$ and $y$ are leaves of $T$, $f_T^*(x)= f_T^*(y)$ and $X$ is a path with $d_X(u)=1$.
\end{lem}
\begin{proof}
Note that
\begin{equation*}
\begin{split}
F(T') & = f_{T'}(x)+ F(T'-x)=f_{T'}(x)+ F(T-x)+F(X-u) \\
      & =f_T(x)f_X(u)+(F(T)-f_T(x))+(F(X)-f_X(u)),\\
F(T'') & = f_{T''}(y)+ F(T''-y)=f_{T''}(y)+ F(T-y)+F(X-u) \\
      & =f_T(y)f_X(u)+(F(T)-f_T(y))+(F(X)-f_X(u)).
\end{split}
\end{equation*}
Hence,
$$
F(T')-F(T'')=(f_T(x)-f_T(y))(f_X(u)-1)> 0,
$$
i.e., $F(T')>F(T'')$.

We partition the set $\mathscr{S}^*(T')$ of leaf containing subtrees of $T'$ as follows:
$$
\mathscr{S}^*(T')=\mathscr{S}^*_1(T') \cup \mathscr{S}^*_2(T') \cup \mathscr{S}^*_3(T'),
$$
where
\begin{itemize}
  \item \text{$\mathscr{S}^*_1(T')=\{\hat{T}:\ \hat{T}$ is a subtree of $T'$ with $x\in V_{\hat{T}}$ and $V_{T'}\cap (PV(T)\setminus \{x\})\not= \emptyset$. \}.}
  \item \text{$\mathscr{S}^*_2(T')=\{\hat{T}:\ \hat{T}$ is a subtree of $T'$  with $x\in V_{\hat{T}}$, $V_{T'}\cap (PV(T)\setminus \{x\})= \emptyset$, $V_{T'}\cap (PV(X)\setminus \{u\})\not= \emptyset$\}.}

  \item \text{$\mathscr{S}^*_3(T')=\{\hat{T}:\ \hat{T}$ is a subtree of $T'$ with $x\notin V_{\hat{T}}$, $V_{T'}\cap PV(T')\not= \emptyset$ \}.}
\end{itemize}
Then we have
$$
|\mathscr{S}^*_1(T')|=f_X(u)f_T^*(x), \  \  \  |\mathscr{S}^*_2(T')|=f_X^*(u)(f_T(x)-f_T^*(x))
$$
and
\begin{equation*}
|\mathscr{S}^*_3(T')|=F^*(T-x)+F^*(X-u)=\left \{ \begin{aligned}
 F^*(T)-f_T(x)+F^*(X-u), &   \ \ \ \ x\in PV(T),\\
 F^*(T)-f_T^*(x)+F^*(X-u), &   \ \ \ \ x\not\in PV(T),
\end{aligned} \right.
\end{equation*}
Hence,
\begin{equation}
F^*(T')=f_X(u)f_T^*(x)+f_X^*(u)(f_T(x)-f_T^*(x))+\left\{ \begin{aligned}
 F^*(T)-f_T(x)+F^*(X-u), &   \ \ \ \ x\in PV(T),\\
 F^*(T)-f_T^*(x)+F^*(X-u), &   \ \ \ \ x\not\in PV(T),
\end{aligned} \right.
\end{equation}

Similarly,
\begin{equation}
F^*(T'')=f_X(u)f_T^*(y)+f_X^*(u)(f_T(y)-f_T^*(y))+\left\{ \begin{aligned}
 F^*(T)-f_T(y)+F^*(X-u), &   \ \ \ \ y \in PV(T),\\
 F^*(T)-f_T^*(y)+F^*(X-u), &   \ \ \ \ y \not\in PV(T),
\end{aligned} \right.
\end{equation}

First consider that neither $x$ nor $y$ is a leaf of $T$, then (2.3) and (2.4) give
\begin{equation*}
\begin{split}
  F^*(T')-F^*(T'') & =(f_T^*(x)-f_T^*(y))f_X(u)+f_X^*(u)(f_T(x)-f_T^*(x)-f_T(y)+f_T^*(y))-f_T^*(x)+f_T^*(y) \\
                   & =(f_T^*(x)-f_T^*(y))(f_X(u)-f_X^*(u)-1)+f_X^*(u)(f_T(x)-f_T(y))>0.
\end{split}
\end{equation*}

Next consider $y$ is  a leaf while $x$ is not a leaf of $T$, then in view of (2.3) and (2.4) we have
\begin{equation*}
\begin{split}
  F^*(T')-F^*(T'') & =(f_T^*(x)-f_T^*(y))f_X(u)+f_X^*(u)(f_T(x)-f_T^*(x)-f_T(y)+f_T^*(y))-f_T^*(x)+f_T(y) \\
                   & >(f_T^*(x)-f_T^*(y))f_X(u)+f_X^*(u)(f_T(x)-f_T^*(x)-f_T(y)+f_T^*(y))-f_T^*(x)+f_T^*(y) \\
                   & =(f_T^*(x)-f_T^*(y))(f_X(u)-f_X^*(u)-1)+f_X^*(u)(f_T(x)-f_T(y))>0.
\end{split}
\end{equation*}

Now consider that both $x$ and $y$ are leaves of $T$, then (2.3) and (2.4) give
\begin{align}
  F^*(T')-F^*(T'') & =(f_T^*(x)-f_T^*(y))f_X(u)+f_X^*(u)(f_T(x)-f_T^*(x)-f_T(y)+f_T^*(y))-f_T(x)+f_T(y)\notag \\
                   & =(f_T^*(x)-f_T^*(y))(f_X(u)-f_X^*(u))+(f_X^*(u)-1)(f_T(x)-f_T(y))\notag\\
                   &\geqslant 0.\label{eq:2.5}
\end{align}

Note that $f_X(u)>f_X^*(u)$ and $f_T(x)>f_T(y)$, hence the equality holds in (\ref{eq:2.5}) if and only if $f_T^*(x)=f_T^*(y)$ and $f_X^*(u)=1$.
Thus, $F^*(T')=F^*(T'')$ if and only if both $x$ and $y$ are leaves of $T,\, f_T^*(x)=f_T^*(y)$ and $f_X^*(u)=1$ with $X$ is a path and $u$ is a pendant vertex of $X$.
\end{proof}
\begin{lem}\label{lem2.4}
Given an $n$-vertex path $P_n=v_1 v_2 \ldots v_n$, one has
\begin{kst}
\item[{\rm (i)}]{\rm (\cite{L-W-J})} $f_{P_n}(v_k)=f_{P_n}(v_{n-k+1})=k(n-k+1), k\in \{1,2,\ldots, n\}$ and
$
f_{P_n}(v_1)< f_{P_n}(v_2)<\cdots <f_{P_n}(v_i)<f_{P_n}(v_{i+1}) <\cdots < f_{P_n}(v_{\lfloor
\frac{n+1}{2} \rfloor})=f_{P_n}(v_{\lceil \frac{n+1}{2} \rceil}).
$

\item[{\rm (ii)}]$f_{P_n}^*(v_1)=f_{P_n}^*(v_n)=1, f_{P_n}^*(v_k)=n$ for $k\in \{2,\ldots, n-1\}$.
\end{kst}
\end{lem}
\begin{proof}
(ii) follows directly by the the definition of $f_T^*(v)$.
\end{proof}

By Lemmas \ref{lem2.3} and \ref{lem2.4}, the following lemma follows immediately.
\begin{lem}\label{lem2.5}
Given a tree $T$ with at least two vertices and a path $P_k= v_1 v_2 \ldots
v_k$, let $T_i$ be a tree obtained from $T$ and $P_k$ by identifying one vertex of $T$ with
$v_i$ of $P_k$. Then
$F(T_i)=F(T_{k-i+1}), F^*(T_i)=F^*(T_{k-i+1}),$
$
F(T_1)< F(T_2)<\cdots <F(T_i)  <\cdots < F(T_{\lfloor \frac{k+1}{2}\rfloor}),$ and $F^*(T_1)< F^*(T_2)<\cdots <F^*(T_i)  <\cdots < F^*(T_{\lfloor \frac{k+1}{2}\rfloor}).
$
\end{lem}
\begin{lem}\label{lem2.6}
Given a tree $T$ with $uv\in E_T$ and $d_T(u)=1$, one has
\begin{wst}
\item[{\rm (i)}]{\rm (\cite{L-W-J})} $f_T(u)\leqslant f_T(v)$ with equality if and only if $T\cong K_2$.

\item[{\rm (ii)}]$f_T^*(u)\leqslant f_T^*(v)$ with equality if and only if $T\cong K_2$.
\end{wst}
\end{lem}
\begin{proof}
If $T\cong K_2,$ it's routine to check that $f_T^*(u)= f_T^*(v)=1$. In what follows, we consider that $T\not\cong K_2.$
Note that $uv$ is a pendant edge, the map $f: \, \mathscr{S}^*(T, u)\rightarrow \mathscr{S}^*(T-u, v)$ that sends each $T$ to $T-u$ is a bijection.
On the other hand, by Lemma \ref{lem2.1}, we have $|\mathscr{S}^*(T-u, v)|<|\mathscr{S}^*(T, v)|$, i.e., $f_{T-u}^*(v)<f_T^*(v)$, hence
our results follows immediately.
\end{proof}

\section{\normalsize Three transformations on trees}\setcounter{equation}{0}

In this section, we introduce three transformations on trees, which will be used to prove our main results.
\begin{defi}
Let $T'$ (resp. $T''$) be a tree with $u\in V_{T'}$ (resp. $v\in V_{T''}$), where $|V_{T''}| = r + 1\geqslant2$. Let $T_1$ be a tree obtained from $T'$ and $T''$ by identifying vertices $u$ with $v$; see Fig 3. In particular, if $T''\cong P_{r+1}$ with $v$ being an endvertex, we may denote the resultant graph by $T_2$; see Fig 3.  We say that $T_2$ is an $A$-transformation of $T_1$ on $T''$.
\end{defi}
\begin{figure}[h!]
\begin{center}
\psfrag{a}{$T'$}\psfrag{b}{$T''$}\psfrag{c}{$u=v$}
\psfrag{d}{$T_1$}\psfrag{g}{$T_2$}
\psfrag{e}{$v$}\psfrag{f}{$P_{r+1}$}
  \includegraphics[width=90mm]{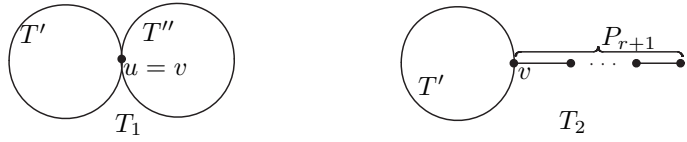}\\
  \caption{Trees $T_1$ and $T_2$.}
\end{center}
\end{figure}
\begin{lem}\label{lem3.1}
Let $T_1$ and $T_2$ be the trees defined as above. Then
\begin{wst}
\item[{\rm (i)}]{\rm (\cite{25})}
$
F(T_1)\geqslant F(T_2)
$
with equality if and only if $T''= P_{r+1}$ with $d_{T''}(v)=1.$
\item[{\rm (ii)}]
$
F^*(T_1)\geqslant F^*(T_2)
$
with equality if and only if $T''= P_{r+1}$  with $d_{T''}(v)=1.$
\end{wst}
\end{lem}
\begin{proof}
In view of the proof of Lemma \ref{lem2.3}, let $T'$ be $X$ and $T''$ (resp. $P_{r+1}$) be $T$ in Lemma \ref{lem2.3}. Then we have
\begin{align}
 F^*(T_2)=& f_{T'}(u)f_{P_{r+1}}^*(v)-f_{T'}^*(u)(f_{P_{r+1}}(v)-f_{P_{r+1}}^*(v)) + F^*(P_{r+1})-f_{P_{r+1}}(v)+F^*(T'-u) \notag\\
   =&  (f_{T'}(u)-f_{T'}^*(u))+(f_{T'}^*(u)-1)(r+1) + F^*(P_{r+1})+F^*(T'-u),\label{eq:3.1}\\
  F^*(T_1) =& f_{T'}(u)f_{T''}^*(v)+f_{T'}^*(u)(f_{T''}(v)-f_{T''}^*(v))+\left\{ \begin{aligned}
 F^*(T'')-f_{T''}(v)+F^*(T'-u) &   \ \ \ \ v \in  PV(T''),\\
F^*(T'')-f_{T''}^*(v)+F^*(T'-u) &   \ \ \ \ v \not\in  PV(T'')
  \end{aligned} \right. \notag \\
   =& (f_{T'}(u)-f_{T'}^*(u))f_{T''}^*(v)+f_{T'}^*(u)f_{T''}(v)+
  \left\{ \begin{aligned}
   F^*(T'')-f_{T''}(v)+F^*(T'-u) &   \ \ \ \ v \in  PV(T''),\\
F^*(T'')-f_{T''}^*(v)+F^*(T'-u) &   \ \ \ \ v \not\in  PV(T'')
  \end{aligned} \right. \notag \\
  \geqslant & (f_{T'}(u)-f_{T'}^*(u))f_{T''}^*(v)+(f_{T'}^*(u)-1)f_{T''}(v)+F^*(T'')+F^*(T'-u)\label{eq:3.2}\\
  \geqslant & (f_{T'}(u)-f_{T'}^*(u))+(f_{T'}^*(u)-1)f_{T''}(v)+F^*(T'')+F^*(T'-u)\label{eq:3.3}\\
  \geqslant & (f_{T'}(u)-f_{T'}^*(u))+(f_{T'}^*(u)-1)(r+1) + F^*(P_{r+1})+F^*(T'-u)\label{eq:3.4}\\
  =& F^*(T_2).\label{eq:3.5}\notag
\end{align}
Equality holds in (\ref{eq:3.2}) if and only if $v$ is a leaf of $T''.$ Note that $f_{T'}(u)>f_{T'}^*(u)$, hence equality holds in
(\ref{eq:3.3}) if and only if $f^*_{T''}(v)=1$, i.e., $v$ is a leaf of a path; Equality holds in (\ref{eq:3.4}) if and only if $T''\cong P_{r+1}$ and $v$ is a leaf. The last equality follows by (\ref{eq:3.1}). Hence, $F^*(T_1)= F^*(T_2)$ if and only if $T''= P_{r+1}$ and $v$ is one of its endvertices, as desired.
\end{proof}

\begin{defi}
Let $T_1$ be the graph as depicted in Fig. 4, where $T'$ (resp. $T''$) is a tree with at least two vertices. Let $\hat{T}_2$ be a tree obtained
from $T_1$ by deleting the edge $uv$ and identifying its endvertices. Let $T_2$ be the tree obtained from $\hat{T}_2$ by attaching an pendant edge to $u;$ see Fig. 4.  We call the procedure constructing
$T_2$ from $T_1$ the $B$-transformation of $T_1$.
\end{defi}
\begin{figure}[h!]
\begin{center}
  \psfrag{u}{$u$}\psfrag{t}{${\small\bullet}$} \psfrag{v}{$v$}\psfrag{d}{$T'$} \psfrag{e}{$T''$}
\psfrag{a}{$T_1$}\psfrag{b}{$T_2$}\psfrag{c}{$u=v$}\psfrag{k}{$\hat{T}_2$}
  \includegraphics[width=120mm]{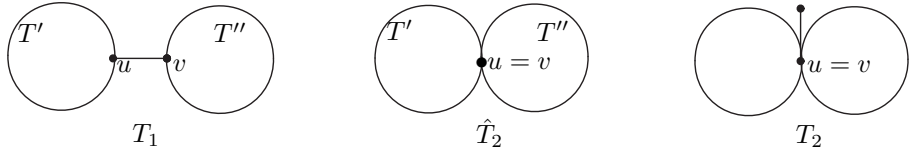}\\
  \caption{Trees $T_1, \hat{T}_2$ and $T_2$.}
\end{center}
\end{figure}
By Lemmas \ref{lem2.3} and \ref{lem2.6}, the following lemma follows immediately.
\begin{lem}\label{lem3.2}
Let $T_1$ and $T_2$ be the trees defined as above, we have $F(T_1)<F(T_2)$ and $F^*(T_1)<F^*(T_2)$.
\end{lem}
\begin{figure}[h!]
\begin{center}
  \psfrag{u}{$u$} \psfrag{v}{$v$}\psfrag{d}{$T_m$} \psfrag{e}{$\vdots$}
  \psfrag{1}{$v_1$} \psfrag{2}{$v_2$}\psfrag{3}{$v_{m-1}$} \psfrag{4}{$v_m$}
\psfrag{a}{$T_1$}\psfrag{b}{$T_2$}\psfrag{c}{$T_{m-1}$}
\psfrag{z}{$v$} \psfrag{y}{$w$}\psfrag{x}{$x$} \psfrag{8}{$T$}\psfrag{9}{$T_0$}
  \includegraphics[width=140mm]{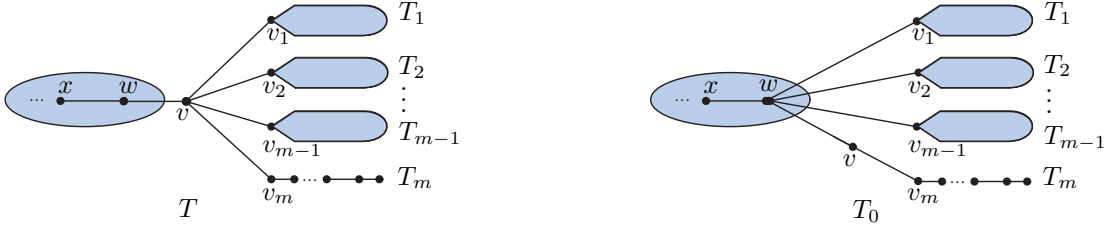}\\
  \caption{$C$-transformation on $v$.}
\end{center}
\end{figure}

\begin{defi}
Let $T$ be an arbitrary tree, rooted at a center vertex $u$ and let $v$ be a vertex with degree $m+ 1$. Suppose that $wv\in E_T$ and $P_T(u,w)\subset P_T(u,v)$(we call $w$ the parent of $v$ in $T$) and that $T_1, T_2,\dots ,T_m$ are subtrees under $v$ with root vertices
$v_1, v_2,\dots v_m$, such that the tree $T_m$ is actually a path. We form a tree $T_0$ by removing the edges $vv_1, vv_2, \dots vv_{m-1}$ from
$T$ and adding new edges $wv_1,wv_2,\dots wv_{m-1}$; see Fig. 5. If $v$ is not a center vertex, we say that $T_0$ is a $C$-transformation of $T$. And if $w$ and $v$ are both center of $T$ with $d_T(w)>2$, we say that $T_0$ is a $C'$-transformation of $T$.

\end{defi}
This transformation preserves the number of pendant vertices in a tree $T$, and does not increase its diameter.
\begin{lem}\label{lem3.3}
Let $T$ and $T_0$ be the trees defined as above, we have $F(T)<F(T_0)$ and $F^*(T)<F^*(T_0)$.
\end{lem}
\begin{proof}
Let $W$ be the component that contains $v$ in $T-\{vv_1, vv_2, \ldots, vv_{m-1}\}$ and $X$ be the component that contains $v$ in $T-\{w, v_m\}$. Now we consider $f_W(w), f_W(v), f_W^*(w)$ and  $f_W^*(v)$. It is routine to check that
$$
f_W(w)=f_{W-v}(w)+f_{W}(w*v), f_W(v)=f_{W-w}(v)+f_{W}(w*v).
$$
Hence,
$$
f_W(w)-f_W(v)=f_{W-v}(w)-f_{W-w}(v)=f_{W-v}(w)-|V_{T_m}|-1.
$$

 If $v$ is not a center of $T$ and its parent is $w$, then there is a proper subtree of the component that contains $w$ in $W-v$, say $T'$, with $T'\cong P_{|V_{T_m}|+1}$. Hence, we have $f_{W-v}(w)> f_{T'}(w) \geqslant |V_{T_m}|+1,$ i.e., $f_W(w) > f_W(v)$. By Lemma \ref{lem2.3} we have $F(T_0) > F(T).$

Note that neither $w$ nor $v$ is a leaf of $W$, hence by a similar discussion as above we also have
$$
f_W^*(w)-f_W^*(v)=f_{W-v}^*(w)-f_{W-w}^*(v)= f_{W-v}^*(w)-1 \geqslant 0,
$$
i.e., $f_W^*(w) \geqslant f_W^*(v)$.
By Lemma \ref{lem2.3} we have $F^*(T_0) > F^*(T)$.
If $w=u$ is the center of $T$ with $d_T(w)>2$, then we can also have a proper subtrees of the component that contains $w$ in $W-v$ say $T''$ with $T''\cong P_{|V_{T_m}|+1}$. By a similar discussion as above, we can also have $F(T_0) > F(T), F^*(T_0) > F^*(T).$
This completes the proof.
\end{proof}

\section{\normalsize Enumeration of subtrees of some types of trees}\setcounter{equation}{0}

In this section, we determine sharp upper (or lower) bound on the total number of subtrees (or leaf containing subtrees) of some type of trees.

The \textit{matching number} of a graph $G$ is the maximum size of an independent (pair-wise
nonincident) set of edges of $G$ and will be denoted by $q(G)$. Let $\mathscr{M}_{n,q}$ be the set of all $n$-vertex trees with matching number $q$. Let $A(n,q)$ be the tree that is obtained by attaching $q-1$ pendant edges to $q-1$ pendant vertices of the star $K_{1, n-q}$. It is routine to check that $A(n,q)\in \mathscr{M}_{n,q}$. Given a vertex $w$ in $G$, call $w$ a \textit{perfectly matched vertex} if it is matched in any maximum matching of $G$.
\begin{thm}\label{thm4.1}
Among $\mathscr{M}_{n,q}$ precisely the graph $A(n,q)$, which has $2^{n-2q+1}\cdot3^{q-1}+n+q-2$
subtrees, maximizes the total number of subtrees and has $2^{n-2q+1}\cdot3^{q-1}-2^{q-1}+n-1$ leaf containing subtrees, maximizes the total number of leaf containing subtrees.
\end{thm}
\begin{proof}
Choose $T$ in $\mathscr{M}_{n,q}$ such that its total number of subtrees (resp. leaf containing subtrees) is as large as possible. If $T$ contains a pendant path of length $p > 2$, say $v_1v_2v_3\ldots v_pv$ with
$v_1\in PV(T)$ and $d_T(v)\geqslant 3$,  then
$f_{T-v_2-v_1}(v_3)<f_{T-v_2-v_1}(v_4),\, f^*_{T-v_2-v_1}(v_3)<f^*_{T-v_2-v_1}(v_4)$ by Lemma 2.7.
Let
$$
T_0=T-v_2v_3+v_2v.
$$
It is routine to check that $T_0$ is in $\mathscr{M}_{n,q}$. By Lemma \ref{lem2.3} we get $F(T)<F(T_0), F^*(T)<F^*(T_0)$, a contradiction. Hence, any pendant path contained in $T$ is of length at most 2.

If there exists a non-center vertex $v\in V_T$ such that $T$ contains $r$ pendant edges and $s$ pendant paths of length 2 attached to $v$, then assume that $w$ is the parent of $v$.
Consider the following possible cases.

\medskip\noindent
$\bullet$  $s=0$ and $w$ is perfectly matched. Apply $C$-transformation at $v$ once, and get $r-1$ pendant edges and one
pendant path $P_3$ attached at $w$ in the resultant graph, say $\hat{T}$. It is routine to check that $\hat{T}$ is in $\mathscr{M}_{n,q}$. By Lemma \ref{lem3.3} we get $F(T)<F(\hat{T}), F^*(T)<F^*(\hat{T})$, a contradiction.

\medskip\noindent
$\bullet$  $s=0$ and $w$ is not perfectly matched.  Applying $B$-transformation at the edge $wv$, we get
$r+1$ pendant edges at $w$ in the resultant graph, say $\hat{T}$. It is routine to check that $\hat{T}$ is in $\mathscr{M}_{n,q}$. By Lemma \ref{lem3.2} we get $F(T)<F(\hat{T}), F^*(T)<F^*(\hat{T})$, a contradiction.

\medskip\noindent
$\bullet$  $r=0$. Applying $C$-transformation at $v$, we get $s-1$ pendant paths $P_3$'s and one pendant path $P_4$
attached at $w$ in the resultant graph, say $\hat{T}$. It is routine to check that $\hat{T}$ is in $\mathscr{M}_{n,q}$. By Lemma \ref{lem3.3} we get $F(T)<F(\hat{T}), F^*(T)<F^*(\hat{T})$, a contradiction.

\medskip\noindent
$\bullet$  $r>0, s>0$ and $w$ is perfectly matched. Applying $C$-transformation at $v$, we get $r-1$ pendant edges and $s + 1$
pendant paths $P_3$'s attached at $w$ in the resultant graph, say $\hat{T}$. It is routine to check that $\hat{T}$ is in $\mathscr{M}_{n,q}$. By Lemma \ref{lem3.3} we get $F(T)<F(\hat{T}), F^*(T)<F^*(\hat{T})$, a contradiction.

\medskip\noindent
$\bullet$  $r>0, s>0$ and $w$ is not perfectly matched. Applying $B$-transformation at the edge $wv$, we get
$r+1$ pendant edges and $s$ pendant paths $P_3$'s attached at $w$ in the resultant graph, say $\hat{T}$. It is routine to check that $\hat{T}$ is in $\mathscr{M}_{n,q}$. By Lemma \ref{lem3.2} we get $F(T)<F(\hat{T}), F^*(T)<F^*(\hat{T})$, a contradiction.

Hence, all the pendant paths of length at most 2 are attached only to the centers of $T$. In order to characterize the structure of $T$, it suffices to show that $T$ contains just one center whose degree is larger than $2$. Otherwise, assume that $T$ contains two centers, say $c_1, c_2$, with $d_T(c_1)>2$ and $d_T(c_2)>2$. Apply$C'$-transformation on $c_1$ in $T$ to get a new tree, say $T'$. It's routine to check that $T'$ is in $\mathscr{M}_{n,q}$. By Lemma 3.3, we get $F(T')>F(T)$ and $F^*(T')>F(T)$, a contradiction.

Therefore, we get that $T\cong A(n,q)$ with the center $c.$ Note that in $A(n,q)$ there exist $n-2q+1$ pendant edges and $q-1$ pendant paths of length 2 attached to $c$, hence
\begin{eqnarray*}
  f_{A(n,q)}(c)&=&2^{n-2q+1}\cdot 3^{q-1},\\
 F(A(n,q)-c)&=&F((n-2q+1)P_1 \cup (q-1)P_2)=n-2q+1+3(q-1)=n+q-2,
\end{eqnarray*}
and $$F^*(A(n,q))=F(A(n,q))-F(A(n,q)-PV(A(n,q)))=F(A(n,q))-F(K_{1,q-1}).$$
This gives
$$
F(A(n,q))=f_{A(n,q)}(c)+F(A(n,q)-c)=2^{n-2q+1}\cdot 3^{q-1}+n+q-2
$$
and
$$   F^*(A(n,q))=2^{n-2q+1}\cdot 3^{q-1}+n+q-2-(2^{q-1}+q-1)=2^{n-2q+1}\cdot 3^{q-1}-2^{q-1}+n-1,$$
as desired.
\end{proof}

Let $\mathscr{S}(n, \gamma)$ be the set of $n$-vertex trees with domination number $\gamma$.
\begin{thm}
Among $\mathscr{S}(n, \gamma),$ the tree $A(n, \gamma)$ maximizes the total number of subtrees (resp. leaf containing subtrees).
\end{thm}
\begin{proof}
It is known from \cite{x-k-f} that $\gamma(G)\leqslant q(G),$ where $q(G)$ is the matching number of $G$. In order to complete the proof, it suffices to show the following claim.
\begin{claim}\label{lem4.2}
If $T_0\in \mathscr{S}(n, \gamma)$ maximizes the total number of subtrees (resp. leaf containing subtrees), then we have $\gamma(T_0)=q(T_0)$.
\end{claim}
\begin{proof}
It suffices to show that $\gamma(T_0)\geqslant q(T_0)$. Otherwise, by the definition of the set $\mathscr{T}(n, \gamma)$, we have $q(T_0) > \gamma(T_0) =\gamma$. Assume that $S =\{v_1, v_2, \ldots, v_{\gamma} \}$ is a dominating set with cardinality $\gamma$. Then there exist $\gamma$ independent edges
$v_1v_1', v_2v_2',\ldots, v_{\gamma}v_{\gamma}'$ in $T_0$. Note that $q(T_0) > \gamma(T_0) = \gamma$, there must exist another edge, say $w_1w_2$, which is independent of each of edges $v_iv_i'$,\, $i = 1, 2, \ldots, \gamma$.
\begin{figure}[h!]
\begin{center}
\psfrag{a}{$v_1$}\psfrag{b}{$v_2$}\psfrag{c}{$v_1'$}\psfrag{d}{$v_2'$}
\psfrag{e}{$w_1$}\psfrag{f}{$w_2$}\psfrag{g}{$v_{\gamma}$}\psfrag{h}{$v_{\gamma}'$}\psfrag{A}{$T_0$}
  \psfrag{B}{$T_0'$}
  \includegraphics[width=120mm]{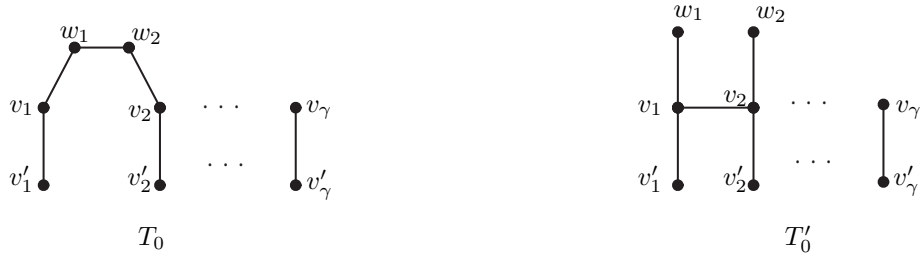}\\
  \caption{The structures of $T_0$ and $T_0'$ in Claim \ref{lem4.2} }
\end{center}
\end{figure}

If the two vertices $w_1, w_2$ is dominated by the same vertex $v_i \in S$, then a triangle $C_3 = w_1w_2v_i$
occurs. This is impossible because of the fact that $T_0$ is a tree. Therefore $w_1, w_2$ are dominated by two different vertices from $S$.
Without loss of generality, assume that $w_i$
is dominated by the vertex $v_i$ for $i = 1, 2$ (see Fig. 6). Now we construct a new tree $T_0'\in \mathscr{T}(n, \gamma)$ by
$B$-transformation of $T_0$ on the edges $v_1w_1$ and $v_2w_2$, respectively. By Lemma \ref{lem3.2}, we have $F(T_0)< F(T_0'), F^*(T_0)<F^*(T_0')$, a contradiction. This completes the proof of Claim~1.
\end{proof}

Theorem 4.2 follows immediately from Theorem 4.1 and Claim 1.
\end{proof}

Let $H\circ K_1$ be the graph obtained by attaching a leaf to each of the vertices of the graph $H$.
\begin{thm}
Let $T\in \mathscr{S}(n, \frac{n}{2})(n\geqslant 4)$, then
\begin{equation}
F(T)\geqslant 2^{\frac{n}{2}+2}-\frac{n}{2}-4,\ \ \ \ \ F^*(T)\geqslant 2^{\frac{n}{2}+2}-\frac{n}{2}-4-{{\frac{n}{2}+1}\choose {2}}.\label{eq:4.1}
\end{equation}
Each of the equalities in (\ref{eq:4.1}) holds if and only if $T\cong P_{\frac{n}{2}}\circ K_1.$
\end{thm}
\begin{proof}
It is known \cite{3,13} that if $n=2\gamma$, then a tree $T$ belongs to $\mathscr{S}(n, \gamma)$ if and only if there exists a tree $H$ of order $\gamma=\frac{n}{2}$ such that $T = H\circ K_1$. Hence it suffices to show the following fact.
\begin{fac}
For any tree $T,$ one has
\begin{equation}\label{eq:0042}
F(T\circ K_1)\geqslant F(P_{|V_T|}\circ K_1), \ \ \ F^*(T\circ K_1) \geqslant F^*(P_{|V_T|}\circ K_1).
\end{equation}
Each of the equalities in (\ref{eq:0042}) holds if and only if $T\cong P_{|V_T|}.$
\end{fac}
\begin{proof}
For any $u$ in $V_T$ and $1\leqslant m\leqslant |V_T|$,
let $\mathscr{S}^m(T; u)$ denote the set of all $m$-vertex subtrees of a tree $T$ each of which contains $u.$ It is routine to check that
\begin{eqnarray}
  F(T\circ K_1)&=& \sum_{T_1\in \mathscr{S}(T)}2^{|V_{T_1}|}+|V_T|\label{eq:4.02} \\
               &=& \sum_{T_1\in \mathscr{S}(T-u)}2^{|V_{T_1}|}+\sum_{T_1\in \mathscr{S}(T;u)}2^{|V_{T_1}|}+|V_T| \notag\\
               &=& \sum_{T_1\in \mathscr{S}(T-u)}2^{|V_{T_1}|}+\sum_{m=1}^{|V_T|}|\mathscr{S}^m(T;u)|2^m+|V_T|\label{eq:4.3}
\end{eqnarray}
and
\begin{eqnarray}
  F^*(T\circ K_1)&=&F(T\circ K_1)-F(T)\notag\\
               &=& \sum_{T_1\in \mathscr{S}(T)}(2^{|V_{T_1}|}-1)+|V_T|\label{eq:4.5} \\
               &=& \sum_{T_1\in \mathscr{S}(T-u)}(2^{|V_{T_1}|}-1)+\sum_{m=1}^{|V_T|}|\mathscr{S}^m(T;u)|(2^m-1)+|V_T|.\label{eq:4.6}
\end{eqnarray}

Assume that $T\not\cong P_{|V_T|}$. If $|V_T|=2$ or $3$, our result is clearly true. If $|V_T|= 4$, there exist only two
trees, i.e., $P_4$ and $K_{1,3}$, hence $T=K_{1,3}$. In this case, for any $u\in PV(T)$ we have
\begin{equation}\label{eq:4.7}
    |\mathscr{S}^1(T;u)|=|\mathscr{S}^2(T;u)|=|\mathscr{S}^4(T;u)|=1, |\mathscr{S}^3(T;u)|=2
\end{equation}
And for any $v\in PV(P_4)$, we have
\begin{equation}\label{eq:4.8}
    |\mathscr{S}^1(P_4;v)|=|\mathscr{S}^2(P_4;v)|=|\mathscr{S}^3(P_4;v)|=|\mathscr{S}^4(P_4;u)|=1.
\end{equation}
Note that $P_4-u=K_{1,3}-v$, hence by (\ref{eq:4.3}), (\ref{eq:4.6})-(\ref{eq:4.8}) we have
$$
F(K_{1,3}\circ K_1)>F(P_4\circ K_1), F^*(K_{1,3}\circ K_1)>F^*(P_4\circ K_1).
$$

In what follows we assume that the inequalities hold in (\ref{eq:0042}) for all trees of order less than $|V_T|$. On the one hand, for any $u\in PV(T)$ and $v\in PV(P_{|V_T|})$, we have
\[\label{eq:0009}
F((T-u)\circ K_1)\geqslant F((P_{|V_T|}-v)\circ K_1), F^*((T-u)\circ K_1)\geqslant F^*((P_{|V_T|}-v)\circ K_1),
\]
Each of the equalities in (\ref{eq:0009}) holds if and only if $T-u\cong P_{|V_T|}-v$.
Hence by (\ref{eq:4.02}) and (\ref{eq:4.5}), we have
\begin{equation}\label{4.9}
\sum_{T_1\in \mathscr{S}(T-u)}2^{|V_{T_1}|}\geqslant \sum_{T_1\in \mathscr{S}(P_{|V_T|}-v)}2^{|V_{T_1}|},\sum_{T_1\in \mathscr{S}(T-u)}(2^{|V_{T_1}|}-1)\geqslant \sum_{T_1\in \mathscr{S}(P_{|V_T|}-v)}(2^{|V_{T_1}|}-1).
\end{equation}
On the other hand, it is easy to see that for any $w\in PV(T)\setminus \{u\}$, $T-w\in \mathscr{S}^{|V_T|-1}(T;u)$, so we have
\begin{equation}\label{4.10}
   |\mathscr{S}^{|V_T|-1}(T;u)|>1=|\mathscr{S}^{|V_T|-1}(P_{|V_T|};v)|.
\end{equation}
Furthermore, for $m=1,2,\ldots, |V_T|-2, |V_T|$,
\begin{equation}\label{4.11}
   |\mathscr{S}^{m}(T;u)|\geqslant 1=|\mathscr{S}^{m}(P_{|V_T|};v)| .
\end{equation}
Hence, (\ref{eq:0042}) follows by (\ref{eq:4.3}),(\ref{eq:4.6}),(\ref{4.9})-(\ref{4.11}). This completes the proof of Fact 1.
\end{proof}

Note that $|\mathscr{S}^{m}(P_n;v)|=1$ for $m=1,2\ldots,n$, hence by (\ref{eq:4.02}) and (\ref{eq:4.3}) we get
$$
F(P_n\circ K_1)=\sum_{T_1\in \mathscr{S}(P_{n-1})}2^{|V_{T_1}|}+\sum_{m=1}^n2^m+n=F(P_{n-1}\circ K_1)+2^{n+1}-1.
$$
Hence we have
$$
F(P_n\circ K_1)=F(P_1\circ K_1)+\sum_{i=3}^{n+1}2^i-(n-1)=3+\sum_{i=3}^{n+1}2^i-(n-1)=2^{n+2}-n-4.(n\geqslant 2)
$$
and
$$
F^*(P_n\circ K_1)=F(P_n\circ K_1)-F(P_n)=2^{n+2}-n-4-{{n+1}\choose{2}}(n\geqslant 2).
$$

This completes the proof.
\end{proof}

Let $P_k(1^a, 1^b)$ be a tree obtained by attaching $a$ and $b$ pendant vertices to the two endvertices
of $P_k$, respectively.
\begin{thm}
Let $T\in \mathscr{S}(n, 2)$ with $n\geqslant 6$, then
\begin{eqnarray}
  F(T) &\geqslant&  3\cdot\left(2^{\lfloor\frac{n-4}{2}\rfloor}+2^{\lceil\frac{n-4}{2}\rceil}\right)+2^{n-4}+n-1,\label{eq:4.12}\\
   F^*(T) &\geqslant&   3\cdot\left(2^{\lfloor\frac{n-4}{2}\rfloor}+2^{\lceil\frac{n-4}{2}\rceil}\right)+2^{n-4}+n-11.\label{eq:4.13}
\end{eqnarray}
The equality in (\ref{eq:4.12}) (resp. (\ref{eq:4.13})) holds if and only if $T\cong P_4(1^{\lfloor\frac{n-4}{2}\rfloor},1^{\lceil\frac{n-4}{2}\rceil})$.
\end{thm}
\begin{proof}
When $n =6$, this theorem holds as $P_6\in \mathscr{S}(n,2)$. So we only consider the case
when $n\geqslant 7$. Choose $T\in \mathscr{S}(n, 2)$ such that its total number of subtrees (resp. leaf containing subtrees) is as small as possible. Let $S = \{w_1, w_2\}$ be a dominating set of $T$.

If $d_T(w_1,w_2)=1$, $T$ must be the form $P_2(1^a,1^b)$ with $a+b=n-2$. Without loss of generality, assume that $a\leqslant b$. Note that $T'=P_3(1^a,1^{b-1})\in \mathscr{S}(n,2)$, and by Lemma \ref{lem3.2} we have $F(T')<F(T), F^*(T')<F^*(T)$, a contradiction.
By a similar discussion we can show that $d_{T_1}(w_1,w_2)\not=2$. We omit the procedure here.

If $d_{T}(w_1, w_2)\geqslant 4$, then there exists at least one
vertex $x$ on $P_T(w_1, w_2)$ $x$ can not be dominated by $w_1$ or $w_2$, which implies that $T \not\in \mathscr{S}(n, 2)$. Hence we get $d_T(w_1, w_2)= 3$. That is to say, $T\cong P_4(1^a,1^b)$ with $a+b=n-4, 1\leqslant a\leqslant b$. (Note that $P_4(1^0,1^{n-4})=P_3(1^1,1^{n-5})$). Hence, by direct computing we have
\begin{eqnarray*}
  F(P_4(1^a,1^b))&=& 3\cdot(2^a+2^b)+2^{n-4}+n-4+{3\choose{2}}=3\cdot(2^a+2^b)+2^{n-4}+n-1, \\
  F^*(P_4(1^a,1^b))&=& F(P_4(a,b))-F(P_4)=3\cdot(2^a+2^b)+2^{n-4}+n-11.
\end{eqnarray*}
Note that when $a=b-1, b$, our results hold immediately. Hence, we consider $a\leqslant b-2$ in what follows. It is routine to check that
$$
  2^a+2^b>2^{a+1}+2^{b-1}>\cdots>2^{\lfloor\frac{n-4}{2}\rfloor}+2^{\lceil\frac{n-4}{2}\rceil}.
$$
Hence we have
\begin{eqnarray*}
&&  F(P_4(1^a,1^b)) > F(P_4(1^{a+1},1^{b-1}))>\cdots >F(P_4(1^{\lfloor\frac{n-4}{2}\rfloor},1^{\lceil\frac{n-4}{2}\rceil})),\\[5pt] 
 && F^*(P_4(1^a,1^b)) > F^*(P_4(1^{a+1},1^{b-1}))>\cdots >F^*(P_4(1^{\lfloor\frac{n-4}{2}\rfloor},1^{\lceil\frac{n-4}{2}\rceil})).
\end{eqnarray*}

This completes the proof.
\end{proof}
\begin{thm}
Let $\Delta$ be a positive integer more than two, and let $T$ be an $n$-vertex tree with maximum degree at least $\Delta$. Then
$
F^*(T)\geqslant (n-\Delta+1)\cdot2^{\Delta-1}+\Delta-1.
$
The equality holds if and only if $T\cong T_{n,\Delta}$, where $T_{n,\Delta}$ is obtained from $P_{n-\Delta+1}$ by attaching $\Delta-1$ pendant vertices to one endvertex of $P_{n-\Delta+1}$; see Fig. 7(a).
\end{thm}
\begin{proof}
Choose an $n$-vertex tree $T$ with maximum degree at least $\Delta$ such that its total number of leaf containing subtrees is as small as possible.
Then there exists a vertex $u$ in $V_T$ such that $d_T(u)\geqslant \Delta$. Without loss of generality, we assume that
$\{v_1, v_2, \ldots , v_{\Delta-1}\}\subseteq N_T(u)$. Obviously, the graph $T-\{uv_1, uv_2,\ldots, uv_{\Delta-1}\}$ contains $\Delta$ components $T_1, T_2, \ldots, T_{\Delta-1}, T_{\Delta}$, where $T_i$ contains vertex $v_i$ for $i=1,2,\ldots, \Delta-1$ and $T_{\Delta}$ contains at least two vertices with $u\in V_{T_{\Delta}};$ see Fig. 7(b).
\begin{figure}[h!]
\begin{center}
  \psfrag{1}{$P^1$}\psfrag{2}{$P^2$}\psfrag{3}{$P^3$}\psfrag{5}{$\Delta-1$}
\psfrag{4}{$P^{\Delta}$}\psfrag{A}{(b)\ $T$}\psfrag{B}{(c)\ $T^*$}\psfrag{C}{(a)\ $T_{n,\Delta}$}\psfrag{u}{$u$}\psfrag{9}{(c)}
\psfrag{u}{$u$}\psfrag{v}{$v_1$}\psfrag{c}{$v_2$}
\psfrag{d}{$T_{\Delta}$}\psfrag{a}{$T_1$}\psfrag{b}{$T_2$}\psfrag{c}{$T_{\Delta-1}$}
\psfrag{h}{$T_{\Delta-1}$} \psfrag{x}{$v_1$}\psfrag{y}{$v_2$}\psfrag{z}{$v_{\Delta-1}$}\psfrag{0}{$v_{\Delta}$}
  \includegraphics[width=130mm]{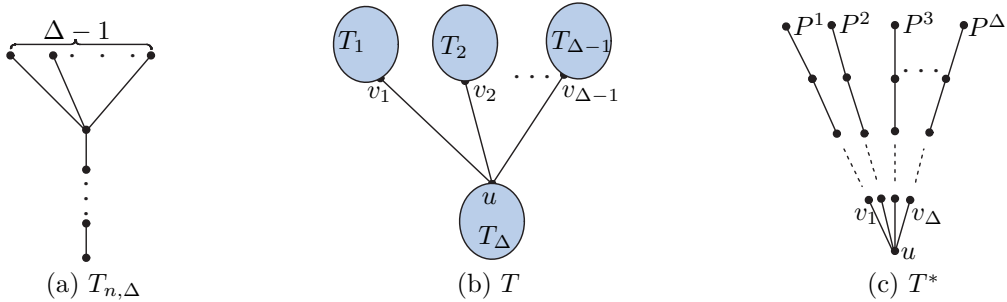}\\
  \caption{Trees $T, T^*$ and $T_{n,\Delta}$ in the proof of Theorem 4.5.} 
\end{center}
\end{figure}

Next we are to show that each $T_i$ is a path for $i=1,2\ldots, \Delta$. In fact, if there exists an $i\in \{1,2\ldots, \Delta\}$ such that
$T_i$ is not a path. Applying $A$-transformations of $T$ on $T_i$ to get a tree, say $\hat{T}$. By Lemma \ref{lem3.1} we have $F^*(\hat{T}) < F^*(T)$, a contradiction. Hence $T\cong T^*$, where $T^*$ is depicted in Fig. 7(c).

Now we show that for any $u_i, u_j \in PV(T)$, $u_iu\in E_T$ or $u_ju\in E_{T}$. In fact, if there exists two pendant vertices, say $u_1, u_2$, such that $u_1u, u_2u \not\in E_T$. Let $P_{T}(u_1,u_2)=u_1w_1w_2\ldots w_ru_2$ with $u, v_1, v_2\in \{w_1,w_2,\ldots, w_r\}$ and $u\not= w_1,w_r$. Let
$T^{**}=T-\{uv_3,uv_4, \ldots, uv_{\Delta}\}+\{w_1v_3,w_1v_4, \ldots, w_1v_{\Delta}\}$. By Lemma \ref{lem2.5}, $F^*(T^{**})<F^*(T)$, a contradiction. Hence, $T\cong T_{n,\Delta}$.

Note that
$
f_{T_{n,\Delta}}(u)=(n-\Delta+1)\cdot2^{\Delta-1},  F(T_{n,\Delta}-u)=F((\Delta-1)P_1 \cup P_{n-\Delta})
$
and $F(H(T_{n,\Delta}))=F(T_{n,\Delta}-PV(T_{n,\Delta}))=F(P_{n-\Delta}),$
hence we have
$$
  F(T_{n,\Delta})  =f_{T_{n,\Delta}}(u)+F(T_{n,\Delta}-u)= (n-\Delta+1)\cdot2^{\Delta-1}+\Delta-1+{n-\Delta+1\choose{2}}.
$$
So we have
\begin{eqnarray*}
   F^*(T_{n,\Delta})&=& F(T_{n,\Delta})- F(H(T_{n,\Delta}))=(n-\Delta+1)\cdot2^{\Delta-1}+\Delta-1,
\end{eqnarray*}
as desired.
\end{proof}
\begin{thm}
Let $\Delta$ be a positive integer more than two, and let $T$ be an $n$-vertex tree with maximum degree at least $\Delta$ having a perfect matching. Then
\begin{eqnarray}
   F(T)&\geqslant & 2(n-2\Delta+3)\cdot3^{\Delta-2}+3\cdot\Delta-5+{n-2\Delta+3\choose{2}},\label{eq:4.1-1}\\
   F^*(T)&\geqslant & 2(n-2\Delta+3)\cdot3^{\Delta-2}-(n-2\Delta+2)\cdot2^{\Delta-2}+n-1.\label{eq:4.1-2}
\end{eqnarray}
Equality holds in (\ref{eq:4.1-1}) (resp. (\ref{eq:4.1-2})) if and only if $T\cong T_{n,\Delta}'$, where $T_{n,\Delta}'$ is the tree obtained from $P_{n-2\Delta+1}$ by attaching $(\Delta-2)\,\, P_3$\!'s and one $P_2$ to one endvertex of $P_{n-2\Delta+3};$ see Fig. 8.
\end{thm}
\begin{figure}[h!]
\begin{center}
  \psfrag{1}{$v_1$}\psfrag{2}{$u_2$}\psfrag{3}{$u_3$}\psfrag{4}{$u_4$}
\psfrag{5}{$u_5$}\psfrag{6}{$u_{\Delta-1}$}
\psfrag{a}{$v_2$}\psfrag{b}{$v_3$}\psfrag{c}{$v_4$}
\psfrag{d}{$v_5$}\psfrag{e}{$v_{\Delta-1}$}
  \includegraphics[width=100mm]{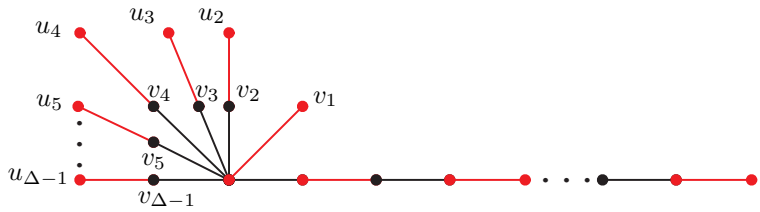}\\
  \caption{Tree $T_{n,\Delta}'$.} 
\end{center}
\end{figure}
\begin{proof}
Choose an $n$-vertex tree $T$ with maximum degree at least $\Delta$ having a perfect matching such that its total number of subtree (resp. leaf containing subtrees) is as small as possible. By a similar discussion as in the proof of Theorem 4.5, we can obtain that $T$ is the graph depicted in Fig. 7(b).

Note that for any two $n$-vertex tree $T_1$ and $T_2$, if $T_2$ is an $A$-transformation of $T_1$, then the maximum matching number of $T_2$ is no less than that of $T_1$. Hence, by a similar discussion as in Theorem 4.5, we have $T\cong T^*$ as depicted in Fig. 7(c) and $T^*$ contains a perfect matching, say $M$. Note that for the vertex $u$ in $T^*$, $u$ is saturated by $M$, hence without loss of generality we assume that $uv_1\in M$. Then we have $|V_{T_1}|, |V_{T_{\Delta}}|$ are odd and $|V_{T_2}|, \ldots |V_{T_{\Delta-1}}|$ are even.

Next we show that $v_1\in PV(T)=\{u_1,u_2\ldots u_{\Delta}\}$.
Suppose that $P_T(u_1, u_{\Delta})=u_1w_1w_2\ldots w_ru_{\Delta}$ with $u,v_1,v_{\Delta}\in \{w_1, w_2\ldots, w_r\}$ and $u\not=w_1,w_r$. Let $\hat{T}=T-\{uv_2, uv_3,\ldots, uv_{\Delta-1}\}+\{w_1v_2, w_1v_3,\ldots, w_1v_{\Delta-1}\}$, $M$ is also a perfect matching of $\hat{T}$. By Lemma \ref{lem2.5},
we have $F(T)>F(\hat{T}), F^*(T)>F^*(\hat{T})$, a contradiction. Hence we have $v_1\in PV(T)$. For convenience, let $u_1:=v_1.$

Now we show that for any $u_i, u_j \in PV(T)\setminus \{u_1\}$, $d_T(u_i,u)=2$ or $d_T(u_j,u)=2$. Note that $T$ contains a perfect matching, hence $d_T(u_i,u)\geqslant 2$ for $u_i \in PV(T)\setminus \{u_1\}$. If there exist two vertices in $PV(T)\setminus \{u_1\}$, say $u_2, u_3$, such that
$d_T(u_2,u)>2,\, d_T(u_3,u)>2$. Denote the unique path connecting $u_2, u_3$ by $P_T(u_2,u_3)=u_2s_1s_2\ldots s_{t-1}s_tu_3$ with $u=s_k$, where $k\not=1,2, t-1, t$. Let $T^{**}=T-\{uv_1,uv_4,\ldots u_{\Delta}\}+\{s_2v_1,s_2v_4,\ldots s_2v_{\Delta}\}$. Note that $M-uv_1+s_2v_1$ is a perfect matching of $T^{**}$. By Lemma \ref{lem2.5}, we have $F(T)>F(T^{**}), F^*(T)>F^*(T^{**})$, a contradiction. So we have $T\cong T_{n,\Delta}'$; see Fig. 8.

Note that
$
f_{T_{n,\Delta}'}(u)=2(n-2\Delta+3)\cdot3^{\Delta-2}, F(T_{n,\Delta}'-u)=F((\Delta-2)P_2 \cup P_1 \cup P_{n-2\Delta+2}),
$
hence
\[
F(T_{n,\Delta}')=2(n-2\Delta+3)\cdot3^{\Delta-2}+3(\Delta-2)+1+{n-2\Delta+3\choose{2}}.
\]
Note that $H(T_{n,\Delta}')=T_{n,\Delta}'-PV(T_{n,\Delta}')=T_{n-\Delta,\Delta-1}.$ Hence,
$$
F(H(T_{n,\Delta}'))=(n-2\Delta+2)\cdot2^{\Delta-2}+\Delta-2+{n-2\Delta+2\choose{2}}.
$$
Together with (4.17) and Fact 1, we have
$$
F^*(T_{n,\Delta}')=F(T_{n,\Delta}')-F(H(T_{n,\Delta}'))=2(n-2\Delta+3)\cdot3^{\Delta-2}-(n-2\Delta+2)\cdot2^{\Delta-2}+n-1,
$$
as desired.
\end{proof}

Let $\mathscr{S}_n^k$ be the set of all $n$-vertex trees with $k$
leaves  ($2\leqslant k\leqslant n-1$). A \textit{spider} is a tree with at most one vertex of degree more than 2, called the \textit{hub} of the spider (if no vertex of degree more than two, then any vertex can be the hub). A \textit{leg} of a spider is a path from the hub to a leaf.
Let $T_n^k$ be an $n$-vertex tree with $k$ legs satisfying all the lengths of $k$ legs, say $l_1, l_2, \ldots, l_k$, are almost equal lengths, i.e.,
$|l_i-l_j|\leqslant 1$ for $1 \leqslant i, j \leqslant k.$ 
It is easy to see that $T_n^k\in \mathscr{S}_n^k$ and $l_i+l_j \in \{2¡¤\lfloor \frac{n-1}{k}
\rfloor , \lfloor \frac{n-1}{k} \rfloor + \lceil \frac{n-1}{k} \rceil ,
2¡¤\lceil \frac{n-1}{k} \rceil \}$, where $1 \leqslant i, j \leqslant k.$

\begin{thm}
Among ${\mathscr S}_n^k$ with $n\geqslant 2$, precisely the graph $T_n^k$, has
$$
\left(\left\lfloor\frac{n-1}{k}\right\rfloor+1\right)^i\left(\left\lceil\frac{n-1}{k}\right\rceil+1\right)^j-\left(\left\lfloor\frac{n-1}{k}\right\rfloor\right)^i\left(\left\lceil\frac{n-1}{k}\right\rceil\right)^j+i\left\lfloor\frac{n-1}{k}\right\rfloor +j\left\lceil\frac{n-1}{k}\right\rceil
$$
leaf containing subtrees, maximizes the total number of leaf containing subtrees, where $i+j=k$ and $n-1\equiv j\pmod{k}$.
\end{thm}
\begin{proof}
Choose $T\in \mathscr{S}_n^k$
such that its total number of leaf containing subtrees is as large as possible. If $k=2$ or, $n-1$, it is easy to see that $\mathscr {S}_n^k=\{T_n^k\}$, our result follows immediately. Hence, in what follows we consider $2<k<n-1$. For convenience, let $W$ be the set of
vertex of degree larger than 2 in $T$.

First we show that for any $v\in W$, $v$ is a center of $T$. Otherwise, apply a $C$ transformation to $v$ of $T$ to get a new tree $T'$. It's straightforward to check that  $T'\in \mathscr{S}_n^k$. By Lemma 3.3, we have $F^*(T)<F^*(T')$, a contradiction to the choice of $T$. Hence,  for any vertex $w\in V_T$ that is not the center of $T$, we have $d_T(w) \leqslant 2$.
If there are two center vertices $c_1$ and $c_2$ in $W$, apply a $C'$-transformation to $c_1$ of $T$ to get a new tree $T'$. Then $T'$ is a spider and by Lemma 3.3 we have $F^*(T')>F^*(T)$, a contradiction.

Now suppose $c$ is the only vertex in $W$. We are to show that for any $u_i, u_j\in PV(T)$, one has
$|d_T(c,u_i)-d_T(c,u_j)| \leqslant 1.$ Assume to the contrary that there exist two pendant vertices, say $u_t, u_l$, in $PV(T)$ such that
\begin{equation}\label{eq:4.18}
|d_T(c,u_t)-d_T(c,u_l)|
\geqslant 2.
\end{equation}
Denote the unique path connecting $u_t$ and $u_l$ by $P_s= w_1 w_2 \ldots w_{i-1}w_iw_{i+1}\ldots
w_s,$ where $w_1=u_t, w_s=u_l$ and $w_i=c, 1 \leqslant i \leqslant s$. In view of (\ref{eq:4.18}), we have
$$
\text{$c=w_i \neq w_{\lfloor\frac{s+1}{2}\rfloor}$\ \ \  and\ \ \ $c=w_i \neq w_{\lceil
\frac{s+1}{2}\rceil}$}.
$$
Hence, by Lemma \ref{lem2.5} there exists an $n$-vertex tree $T' \in \mathscr{S}_n^k$ such that
$F^*(T)<F^*(T')$, a contradiction to the choice of $T$. So we have $T\cong T_n^k$.
Furthermore, we know from {\rm (\cite{L-W-J})} that
\[\label{eq:4.2}
F(T_n^k)=\left(\left\lfloor\frac{n-1}{k}\right\rfloor+1\right)^i\left(\left\lceil\frac{n-1}{k}\right\rceil+1\right)^j+i{\lfloor\frac{n-1}{k}\rfloor+1 \choose{2}}+j{\lceil\frac{n-1}{k}\rceil+1\choose{2}},
\]
where $i+j=k$ and $n-1\equiv j\pmod{k}$.

By Fact 1,
$$
  F^*(T_n^k)=F(T_n^k)-F(T_n^k-PV(T_n^k))=F(T_n^k)-F(T_{n-k}^k).
$$
Hence in view of (\ref{eq:4.2}),
\begin{equation*}
\begin{split}
  F^*(T_n^k)=&
  \left(\left\lfloor\frac{n-1}{k}\right\rfloor+1\right)^i\left(\left\lceil\frac{n-1}{k}\right\rceil+1\right)^j+i{\lfloor\frac{n-1}{k}\rfloor+1 \choose{2}}+j{\lceil\frac{n-1}{k}\rceil+1\choose{2}}\\
&- \left[\left(\left\lfloor\frac{n-1}{k}\right\rfloor\right)^i\left(\left\lceil\frac{n-1}{k}\right\rceil\right)^j+i{\lfloor\frac{n-1}{k}\rfloor \choose{2}}+j{\lceil\frac{n-1}{k}\rceil\choose{2}}\right] \\
    =&\left(\left\lfloor\frac{n-1}{k}\right\rfloor+1\right)^i\left(\left\lceil\frac{n-1}{k}\right\rceil+1\right)^j
    -\left(\left\lfloor\frac{n-1}{k}\right\rfloor\right)^i\left(\left\lceil\frac{n-1}{k}\right\rceil\right)^j+i\left\lfloor\frac{n-1}{k}\right\rfloor +j\left\lceil\frac{n-1}{k}\right\rceil,
\end{split}
\end{equation*}
where $i+j=k$ and $n-1\equiv j\pmod{k}$.

This completes the proof.
\end{proof}
Let $\mathscr{S}_{n,d}$ denote the set of all $n$-vertex trees of
diameter $d$. Let $\hat{T}_{n,k}^d$ be the $n$-vertex tree obtained
from $P_{d+1}=v_1v_2\ldots v_dv_{d+1}$ by attaching $n-d-1$ pendant
edges to $v_k$; see Fig. 9.
\begin{figure}[h!]
\begin{center}
\psfrag{a}{$u_1$}\psfrag{b}{$u_2$}\psfrag{1}{$v_1$}\psfrag{2}{$v_2$}
\psfrag{3}{$v_{k-1}$}\psfrag{4}{$v_k$}\psfrag{5}{$v_{k+1}$}\psfrag{6}{$v_d$}\psfrag{7}{$v_{d+1}$}
  \psfrag{c}{$u_{n-d-1}$}\psfrag{d}{$d$}\psfrag{e}{$k-1$}
  \psfrag{f}{$k+1$}\psfrag{g}{$d+1$}
  \includegraphics[width=60mm]{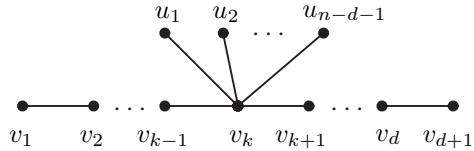}\\
  \caption{Tree $\hat{T}_{n,k}^d.$ }
\end{center}
\end{figure}
\begin{thm}
For any $n\geqslant 2$, precisely the graph $\hat{T}_{n,i}^d$, which has
$$2^{n-d-1}\left(\left\lfloor \frac{d}{2}\right \rfloor+1\right)\left(\left\lceil \frac{d}{2} \right\rceil+1\right)+ {\lfloor \frac{d}{2} \rfloor+1\choose{2}}{\lceil \frac{d}{2} \rceil+1\choose{2}}+n-d-1-{d\choose{2}}$$ leaf containing subtrees, minimizes the total number of leaf containing subtrees among $\mathscr{S}_{n,d}$, where $i=\lfloor\frac{d}{2}\rfloor+1$ or $i=\lceil\frac{d}{2}\rceil+1$.
\end{thm}
\begin{proof}
If $T\in \mathscr{S}_{n,d}$, it is easy to see that $|V_{H(T)}|\geqslant d-1$. By Lemma 2.1 we have $F(H) \geqslant F(P_{d-1})$ with equality if and only if $H\cong P_{d-1}$, which is equivalent to that $T$ is a caterpillar tree of diameter $d$. On the one hand, it is known {\rm (\cite{25})} that $F(T)\leqslant F(\hat{T}_{n,i}^d)$ with equality if and only if $T\cong \hat{T}_{n,i}^d$, where $i=\lfloor\frac{d}{2}\rfloor+1$ or $i=\lceil\frac{d}{2}\rceil+1$.
Together with Fact 1, for any $T\in \mathscr{S}_{n,d}$, we have
\[\label{eq:4-3}
F^*(T)=F(T)-F(H(T))\leqslant F(\hat{T}_{n,i}^d)-F(P_{d-1})=F^*(\hat{T}_{n,i}^d)
\]
with equality if and only if $T\cong \hat{T}_{n,i}^d$ for $i=\lfloor\frac{d}{2}\rfloor+1$ or $i=\lceil\frac{d}{2}\rceil+1$.
On the other hand, it is known {\rm (\cite{L-W-J})} that
$$
F(\hat{T}_{n,i}^d)=2^{n-d-1}\left(\left\lfloor \frac{d}{2}\right \rfloor+1\right)\left(\left\lceil \frac{d}{2} \right\rceil+1\right)+ {\lfloor \frac{d}{2} \rfloor+1\choose{2}}{\lceil \frac{d}{2} \rceil+1\choose{2}}+n-d-1.
$$
Together with (\ref{eq:4-3}), we have
$$F^*(\hat{T}_{n,i}^d)=2^{n-d-1}\left(\left\lfloor \frac{d}{2}\right \rfloor+1\right)\left(\left\lceil \frac{d}{2} \right\rceil+1\right)+ {\lfloor \frac{d}{2} \rfloor+1\choose{2}}{\lceil \frac{d}{2} \rceil+1\choose{2}}+n-d-1-{d\choose{2}}.
$$

This completes the proof.
\end{proof}

\section{\normalsize Concluding remarks}
Du and Zhou \cite{D-Z-Z} characterized the extremal trees with matching number $q$ that minimize the Wiener index; in this paper we show the counterparts of these results for the total number of subtrees of $n$-vertex trees with matching number $q$. In view of Theorem 4.2, we conjecture that there exist the counterparts of these results for the Wiener index
among the $n$-vertex trees with domination number $\gamma$.
Furthermore, for the Wiener index, sharp upper
and lower bounds of trees with given degree sequence are determined; see \cite{18,26,27}. It is natural for us to
determine sharp upper and lower bounds on the total number of subtrees of trees with given degree sequence.
It is difficult but interesting and it is still open. We leave
these problems for future study.

\end{document}